\documentclass[12pt]{amsart}
\usepackage{amsmath,latexsym,amsfonts,amssymb,amsthm}
\usepackage{geometry}
\usepackage{hyperref}
\usepackage{mathrsfs}
\usepackage{graphicx,color}
\usepackage{tikz-cd}
\usepackage{tikz}
\usetikzlibrary{positioning}
\usepackage{mathtools}
\usepackage[all,cmtip]{xy}

\newtheorem{prop}{Proposition}
\newtheorem{thm}[prop]{Theorem}
\newtheorem{cor}[prop]{Corollary}

\newtheorem{lem}[prop]{Lemma}

\theoremstyle{definition}

\newtheorem{defn}[prop]{Definition}
\newtheorem{expl}[prop]{Example}
\newtheorem{rem}[prop]{\it Remark}
\newtheorem{emp}[prop]{}

\numberwithin{equation}{section}

\newcommand{\bP}{\mathbb{P}}
\newcommand{\bC}{\mathbb{C}}
\newcommand{\bR}{\mathbb{R}}

\newcommand{\bQ}{\mathbb{Q}}
\newcommand{\bZ}{\mathbb{Z}}

\newcommand{\tX}{\widetilde{X}}
\newcommand{\tY}{\widetilde{Y}}
\newcommand{\tg}{\widetilde{g}}

\newcommand{\cO}{\mathcal{O}}
\newcommand{\cL}{\mathcal{L}}

\newcommand{\cM}{\mathcal{M}}

\newcommand{\cE}{\mathcal{E}}
\newcommand{\cN}{\mathcal{N}}

\newcommand{\hX}{\hat{X}}

\newcommand{\Supp}{\mathrm{Supp}~}

\newcommand{\mult}{\mathrm{mult}}
\newcommand{\Ext}{\mathrm{Ext}}
\newcommand{\Ex}{\mathrm{Ex}}

\newcommand{\rom}[1]{\lowercase\expandafter{\romannumeral #1\relax}}

\begin{document}

\title{Characterization of projective spaces by Seshadri constants}

\author{Yuchen Liu}
\address{Department of Mathematics, Princeton University, Princeton, NJ, 08544-1000.}
\email{yuchenl@math.princeton.edu}

\author{Ziquan Zhuang}
\address{Department of Mathematics, Princeton University, Princeton, NJ, 08544-1000.}
\email{zzhuang@math.princeton.edu}

\date{\today}

\begin{abstract}
 We prove that an $n$-dimensional complex projective variety is isomorphic to $\bP^n$ if the Seshadri constant of the anti-canonical divisor at some smooth point is greater than $n$. We also classify complex projective varieties with Seshadri constants equal to $n$.
\end{abstract}

\maketitle

\section{Introduction}

It is believed that the projective space $\bP^n$ has the most positive anti-canonical divisor among complex projective varieties. Various characterizations of $\bP^n$ have been found corresponding to different explanations of the ``positivity" of the anti-canonical divisor. Using Kodaira vanishing theorem, Kobayashi and Ochiai \cite{ko73} proved that if an $n$-dimensional projective manifold $X$ with an ample line bundle $H$ satisfies $-K_X\equiv (n+1)H$, then $(X,H)\cong (\bP^n,\cO(1))$. Kobayashi-Ochiai's characterization was generalized by Ionescu \cite{ion86} (in the smooth case) and Fujita \cite{fuj87} (allowing Gorenstein rational singularities) assuming the weaker condition that $K_X+(n+1)H$ is not ample.
Later, Cho, Miyaoka and Shepherd-Barron \cite{cmsb} (simplified by Kebekus in \cite{keb02}) showed that a Fano manifold is isomorphic to $\bP^n$ if the anti-canonical degree of every curve is at least $n+1$. Their proofs rely on deformation of rational curves which still works if we allow isolated local complete intersection quotient singularities (see \cite{ct07}).
Besides, Kachi and Koll\'ar \cite{kk00} gave characterizations of $\bP^n$ in arbitrary characteristic that generalized \cite{ko73} and \cite{cmsb, keb02} with a volume lower bound assumption.

The purpose of this paper is to provide a characterization of $\bP^n$ among complex $\bQ$-Fano varieties by the local positivity of the anti-canonical divisor, namely the \textit{Seshadri constants}. 
Recall that a complex projective variety $X$ is said to be {\it $\bQ$-Fano} if $X$ has klt singularities and $-K_X$ is an ample $\bQ$-Cartier divisor. 

\begin{defn}
 Let $X$ be a normal projective variety and $L$ an ample $\bQ$-Cartier divisor
 on $X$. Let $p\in X$ be a smooth point. The \textit{Seshadri constant} of $L$ at $p$,
 denoted by $\epsilon(L,p)$, is defined as
 \[
  \epsilon(L,p):=\sup\{x\in\bR_{>0}\mid \sigma^*L-xE\textrm{ is ample}\},
 \]
 where $\sigma:\mathrm{Bl}_p X\to X$ is the blow-up of $X$ at $p$, and
 $E$ is the exceptional divisor of $\sigma$.
\end{defn}

It is clear that $\epsilon(-K_{\bP^n},p)=n+1$ for any point $p\in \bP^n$.
Our main result characterizes $\bP^n$ as the only $\bQ$-Fano variety with Seshadri constant bigger than $n$:

\begin{thm}\label{mainthm}
 Let $X$ be a complex $\bQ$-Fano variety of dimension $n$. 
 If there exists a smooth point $p\in X$ such that $\epsilon(-K_X,p)>n$, then $X\cong\bP^n$.
\end{thm}

Note that Theorem \ref{mainthm} only assumes that $\epsilon(-K_X,p)>n$ for \textit{some} smooth point $p$ rather than \textit{any} smooth point (although the existence of such $p$ immediately implies the same inequality for a general smooth point). We also remark here that when $X$ is smooth, Theorem \ref{mainthm} was obtained by Bauer and Szemberg in \cite[Theorem 1.7]{bs} using different methods.

Since the Seshadri constant of a quadric hypersurface in $\bP^{n+1}$ is equal to $n$, the lower bound on the Seshadri constant in Theorem \ref{mainthm} is sharp. It turns out that this is not the only $\bQ$-Fano varieites achieving such lower bound, and the full list is given by the following theorem.

\begin{thm} \label{thm:equality}
Let $X$ be a $n$-dimensional complex $\bQ$-Fano variety. Then there exists a smooth point $p\in X$ with $\epsilon(-K_X,p)=n$ if and only if $X$ is one of the following:
\begin{enumerate}
\item a degree $d+1$ weighted hypersurface $X_{d+1}=(x_0x_{n+1}=f(x_1,\cdots,x_n))\subset\mathbb{P}(1^{n+1},d)$,
\item a quartic weighted hypersurface $X_4=(x_{n+1}^2+x_n h(x_0,\cdots,x_{n-1})=f(x_0,\cdots,x_{n-1}))$ $(h\neq0)$ or $(x_n x_{n+1}=f(x_0,\cdots,x_{n-1}))\subseteq\mathbb{P}(1^n,2, 2)$,
\item the blow-up of $\bP^n$ along the complete intersection of a hyperplane and a hypersurface of degree $d\le n$,
\item the quotient of
the quadric $Q_k=(\sum_{i=0}^k x_i^2=0)\subseteq\mathbb{P}^{n+1}\,(2\leq k\leq n+1)$ by an involution $\tau(x_i)=\delta_i x_i\,(\delta_i=\pm1)$ that
is fixed point free in codimension $1$ and such that not all the $\delta_i(i=0,\cdots,k)$ are the same,
\item a Gorenstein log Del Pezzo surface of degree $\ge4$ \emph{(}for the classification of such surfaces, see \cite[\textsection 3]{anti-canonical}\emph{)}.
\end{enumerate}
\end{thm}

When $X$ is smooth, the condition $\epsilon(-K_X,p)=n$ implies that $(-K_X\cdot C)\ge n$ for any curve $C\subset X$ passing through a very general point $p$. If in addition $X$ has dimension at least 3, then by \cite{miy04} and \cite{cd} $X$ is either a quadric hypersurface or the blow-up of $\bP^n$ along a smooth subvariety of codimension 2 and degree $d\le n$ contained in a hyperplane. On the other hand, in the surface case some of our results have been proved by \cite[Theorem 1.8]{sano} under the somewhat restrictive assumption that $(K_X^2)\in\{4,5,6,7,8,9\}$. Hence the above theorem is a natural generalization of their results to the singular and higher dimensional case, although our proof uses a completely different strategy.

Finally we show that in general the Seshadri constant $\epsilon(-K_{X},p)$ can be any rational number between $0$ and $n$. This is in sharp contrast with Theorem \ref{mainthm} where we have seen that there is a gap between $n$ and $n+1$ for the possible values of $\epsilon(-K_{X},p)$.

\begin{thm}\label{ratsesh}
 For any rational number $0<c\leq n$, there exists an $n$-dimensional $\bQ$-Fano variety $X$ with a smooth point $p$ such that $\epsilon(-K_X,p)=c$.
\end{thm}

\medskip

The paper is organized as follows. In Section \ref{sec2}, we prove Theorem \ref{mainthm}. Denote the blow up of $X$ at $p$ by $\sigma:\hX=\mathrm{Bl}_p X\to X$, then the divisor $D:=\sigma^*(-K_X)-\epsilon(-K_X,p)E$ is nef
by the definition of the Seshadri constant. Under the assumption that $\epsilon(-K_X,p)>n$, we use Kawamata-Viehweg vanishing theorem to show that $D$ is semiample and $g=|kD|:\hX\to Y$ maps $E$ isomorphically onto its image for sufficiently divisible $k$. A simple computation yields that $(-K_{\hX}\cdot C)=\epsilon(-K_X,p)-(n-1)>1$ for any curve $C$ contracted by $g$. We show in Lemma \ref{birlocal} that $g$ cannot be birational under these assumptions and therefore has to be a morphism of fiber type with target  $Y=g(E)\cong \bP^{n-1}$. Then Lemma \ref{ruled} implies that $\hX$ is a $\bP^1$-bundle over $\bP^{n-1}$, thus $X\cong \bP^n$. The proof of Lemma \ref{birlocal} relies on a dimension reduction argument and Lemma \ref{surfacecont}. As an application of Theorem \ref{mainthm}, we show that $\bP^n$ is the only Ding-semistable $\bQ$-Fano variety of volume at least $(n+1)^n$ (see Theorem \ref{ding-pn}). This improves the equality case of \cite[Theorem 1.1]{fuj15} where Fujita proved for Ding-semistable Fano manifolds. 

In Section \ref{sec3}, we classify all $\bQ$-Fano varieties 
with Seshadri constants equal to $n$. By the same reason as
the proof of Theorem \ref{mainthm}, we still have that $D$ 
is semiample. We divide the classification into two parts.
In Section \ref{sec3.1}, we study cases when $g$ is 
birational. We show that $g|_E$ is a closed embedding,
$-(K_Y+g(E))$ is ample, $g(E)$ is nef (see Lemma \ref{lem:g(E)nef}). We 
classify such pairs $(Y,g(E))$ in Lemma \ref{lem:Y}. Then we obtain
the partial classification after a detailed study of the
structure of the birational morphism $g$ (see Lemma \ref{lem:blowup}
and \ref{lem:birational}). In Section 
\ref{sec3.2}, we study cases when $g$ is of fiber type.
It is not hard to see that every fiber of $g$ has dimension $1$, the generic fiber
of $g$ is isomorphic to $\bP^1$, $g|_E:E\to Y$ is a double
cover, and $-K_{\hX}$ is $g$-ample. After pulling
back $g$ to $E$ and taking the normalization, we obtain a 
conic bundle $\tg:\tX\to E\cong\bP^{n-1}$ with two sections
(see Lemma \ref{fiblocal2}, Corollary \ref{fiblocal} and Lemma \ref{basechange}). From 
the classification of the conic bundle $\tg$ and the quotient
map $g|_E$ (see Lemma \ref{disjoint} and \ref{lem:fibertype}), we finish the classification of $X$ and
hence prove Theorem \ref{thm:equality}. Finally in Section \ref{sec4}, we
provide examples showing that the Seshadri constant of a $\bQ$-Fano
variety can be any positive rational number less than $n$.

\subsection*{Acknowledgement}
We would like to thank our advisor J\'anos Koll\'ar for his constant
support, encouragement and numerous inspiring conversations. We would
like to thank Thomas Bauer, Pedro Montero, Tomasz Szemberg and Chenyang Xu for helpful comments. 
The first author also
wishes to thank Xiaowei Wang for useful discussions, and Kento Fujita
for his interest and encouragement. The first author is partially 
supported by NSF grants DMS-0968337 and DMS-1362960.

\section{Proof of Theorem \ref{mainthm}}\label{sec2}

\begin{lem}\label{surfacecont}
Let $\pi:S\rightarrow T$ be a proper birational morphism between normal
surfaces. Let $C\subset S$ be a $K_{S}$-negative $\pi$-exceptional
curve. Then $(-K_{S}\cdot C)\le 1$, with equality if and only if $S$ has only Du Val singularities
along $C$. \emph{(}Since $K_S$ is not necessarily $\bQ$-Cartier, we use the intersection theory of Weil divisors on surfaces by Mumford \cite{mumford}.\emph{)}
\end{lem}

\begin{proof}
Let $\phi:\tilde{S}\rightarrow S$ be the minimal resolution of $S$.
Denote the exceptional curves of $\phi$ by $E_{i}$. Then we have
\[
K_{\tilde{S}}+\sum_i a_{i}E_{i}\equiv\phi^{*}K_{S},\quad\textrm{where }a_{i}\ge 0.
\]
Let $\tilde{C}$ be the birational transform of $C$ under $\phi$.
Since $\pi\circ\phi$ contracts $\tilde{C}$, we have $(\tilde{C}^{2})<0$.
By the assumption that $C$ is $K_S$-negative, we have
\[
(K_{\tilde{S}}\cdot\tilde{C})=(\phi^{*}K_{S}\cdot\tilde{C})-\sum_i a_{i}(E_{i}\cdot\tilde{C})\le(K_{S}\cdot C)<0.
\]
Hence $\tilde{C}$ is a $(-1)$-curve on $\tilde{S}$ and $(-K_{S}\cdot C)\le(-K_{\tilde{S}}\cdot\tilde{C})=1$.

It is clear that $(-K_{S}\cdot C)=1$ if and only if $\sum_i a_{i}(E_{i}\cdot\tilde{C})=0$, i.e. $a_{i}=0$
whenever $\tilde{C}$ intersects $E_{i}$. By the negativity lemma (cf. \cite[Lemma 3.41]{km98}), this is equivalent to
saying that $a_i=0$ whenever $E_i$ is connected to $\tilde{C}$ through a chain of $\phi$-exceptional curves. Thus
the equality holds if and only if $S$ has Du Val singularities along $C$.
\end{proof}

\begin{lem}\label{ruled}
Let $\pi:S\to T$ be a proper surjective morphism from a normal  surface $S$ to a smooth curve $T$. Assume that the generic fiber of $\pi$ is isomorphic to $\bP^1$, and all fibers of $\pi$ are generically reduced and irreducible. Then $\pi$ is a smooth $\bP^1$-fibration, i.e. $S$ is a geometrically ruled surface over $T$.
\end{lem}

\begin{proof}
 For any closed point $t\in T$, denote by $S_t$ the scheme-theoretic fiber of $\pi$ at $t$.
 It is clear that $\pi$ is flat, so $\chi(S_t,\cO_{S_t})=\chi(\bP^1,\cO_{\bP^1})=1$. Besides,
 $S$ being normal implies that the Cartier divisor $S_t$ on $S$ has no embedded points.
 Then $S_t$ being generically reduced and irreducible yields that $S_t$ is an integral curve.
 Therefore, $S_t\cong\bP^1$.
\end{proof}

\begin{emp}[\emph{Proof of Theorem \ref{mainthm}}]
 Denote by $\sigma:\hX=\mathrm{Bl}_p X\to X$ the blow up of $X$ at $p$ with
 exceptional divisor $E$. Let $D:=\sigma^*(-K_X)-\epsilon(-K_X,p)E$
 be the nef divisor. Since $-K_{\hX}=\sigma^*(-K_X)-(n-1)E$, we know
 that $D-K_{\hX}$ is ample. Hence Shokurov's basepoint-free theorem \cite[Theorem 3.3]{km98}
 implies that $D$ is semiample. 
 
 Let $g:\hX\to Y$ be the ample model of $D$ (i.e. $g$ is the morphism determined by the complete linear system $|kD|$ for some $k\gg 0$). Let $m$ be a positive
 integer such that $mD$ is Cartier. 
 Notice that $mD-E-K_{\hX}$ is ample by 
 $\epsilon(-K_X,p)>n$, so Kawamata-Viehweg vanishing implies that
 $H^1(\hX, mD-E)=0$. Hence $H^0(\hX, mD)\to H^0(E, mD|_E)$ is 
 surjective for $m\in\bZ_{>0}$
 with $mD$ being Cartier. As a result, $g|_E:E\to Y$ is a closed 
 embedding. Thus any curve $C$ contracted by $g$ is not contained in $E$, 
 which implies that $(C\cdot \sigma^*(-K_X))>0$. 
 Since $0=(C\cdot D)=(C\cdot\sigma^*(-K_X))-\epsilon(-K_X,p)(C\cdot E)$, we know that
 $(C\cdot E)>0$.
 
 Suppose $g$ contracts $C$ to a point $y\in Y$. Consider the scheme-theoretic fiber $g^{-1}(y)$ of $g$.
 Since $g|_E$ is a closed embedding, the scheme-theoretic intersection $E\cap g^{-1}(y)$ is a reduced closed point,
 say $q$. If there is another curve $C'\neq C$ contained in $g^{-1}(y)$, then $E\cap g^{-1}(y)$ has multiplicity at least
 $2$ at $q$, a contradiciton! So $\Supp g^{-1}(y)=C$ and $g^{-1}(y)$ is smooth and transversal to $E$ at $q$. In particular,
 we have $(C\cdot E)=1$ for any curve $C$ contracted by $g$. Since $\hat{X}$ has klt singularities, it is Cohen-Macaulay by \cite[Theorem 5.22]{km98}. In addition we have $-K_{\hX}\sim_{g.\bQ.}\lambda E$ where $\lambda=\epsilon(-K_X,p)-n+1>1$. Hence by the following lemma, $g$ cannot be birational.
 
\begin{lem}\label{birlocal}
Let $g:\hat{X}\rightarrow Y$ be a proper birational morphism between
quasi-projective normal varieties and $E$ a smooth $g$-ample Cartier
divisor on $\hat{X}$ such that $-K_{\hat{X}}\sim_{g.\mathbb{Q}.}\lambda E$
for some $\lambda\ge1$. Assume that $\hat{X}$ is Cohen-Macaulay
and $g|_{E}:E\rightarrow G=g(E)$ is an isomorphism, then $\lambda=1$ and
$Y$ is smooth along $G$.
\end{lem}

\begin{proof}
Let $H$ be a very ample divisor on $Y$ such that $H^{0}(Y,\mathcal{O}_{Y}(H))\rightarrow H^{0}(G,\mathcal{O}_{G}(H))$
is surjectve. Let $y\in Y$ be a closed point in the exceptional locus
of $g$ and let $H_{1},\cdots,H_{n-2}$ be general members of $|H|$
containing $y$. Let $C=g^{-1}(y)$ and $S=g^{*}H_{1}\cap\cdots\cap g^{*}H_{n-2}$.
We claim that $S$ is a normal surface. Since $E|_{C}$ is ample and
$g|_{E}$ is an isomorphism, it is easy to see as above that $C$ is an irreducible curve and $E\cap C$
is supported at a single point $q$. As $\hat{X}$ is Cohen-Macaulay,
$S$ is $S_{2}$. By Bertini's theorem $S\backslash C$ is smooth
in codimension one and $G\cap H_{1}\cap\cdots\cap H_{n-2}$ (scheme-theoretic
intersection) is smooth at $y$. It follows that $E|_{S}$ is smooth
at $q$. Since $E$ is Cartier, we see that $S$ is also smooth at
$q\in C$, hence $S$ is smooth in codimension one and it is normal.

It is clear that $g|_{S}$ is a birational morphism that contracts
$C$. By adjuction $K_{S}=(K_{X}+g^{*}H_{1}+\cdots+g^{*}H_{n-2})|_{S}$,
thus $(-K_{S}\cdot C)=(-K_{\hat{X}}\cdot C)=\lambda(E\cdot C)=\lambda\ge1$. On the
other hand by Lemma \ref{surfacecont} we have $(-K_{S}\cdot C)\le1$. Hence $\lambda=(-K_{S}\cdot C)=1$
and $S$ has only Du Val singularities along $C$. Since contracting
a $(-1)$-curve (i.e. a curve that has anti-canonical degree 1) from
a surface with Du Val singularities produces a smooth point, $g(S)$
and hence $Y$ is smooth at $y$. Note that $y$ is arbitrary in the
exceptional locus, so $Y$ is smooth along $G$.
\end{proof}

\begin{rem}
 In fact more is true. Under the same assumptions of the lemma, $\hX$ is indeed the blowup of $Y$ along a divisor in $G$. We postpone its proof to the next section.
\end{rem}

Returning to the proof of Theorem \ref{mainthm}, we see that $g$ has to be a fiber type contraction. Since $g|_E$ is a closed embedding, we know that $g|_E:E\to Y$ is in fact an isomorphism. In particular, $E\cong Y\cong\bP^{n-1}$.
Let us define $S$, $H_i$ as in the proof of Lemma \ref{birlocal}. By the same argument there, $S$ is a normal surface. Since the singular set of $\hX$ has codimension at least $2$, by generic smoothness we know that the generic fiber of $g:\hX\to Y$ is smooth. So the contraction $g$ being $K_{\hX}$-negative implies that the generic fiber of $g$ is a smooth rational curve. In particular, the generic fiber of $g|_S: S\to g(S)$ is isomorphic to $\bP^1$. Hence applying Lemma \ref{ruled} yields that $C\cong\bP^1$, which means that $g:\hX\to Y$ is a smooth $\bP^1$-fibration. 

It is clear that $s=g|_E^{-1}:Y\to E$ gives a section of $g$, thus $\hX=\bP_Y(\cE)$ is a $\bP^1$-bundle where $\cE$ is a rank $2$ vector bundle over $Y$. Then the section $E$ corresponds to a surjection $\cE\twoheadrightarrow\cN$ for some line bundle $\cN$ on $Y$. Denote the kernel of this surjection by $\cM$.
By the adjunction formula on $\bP^1$-bundles, we know that $\cO_Y(-1)\cong s^*N_{E/\hX}\cong \cM^{-1}\otimes\cN$. For simplicity we may assume $\cM\cong\cO_Y$,
then we get $\cN\cong\cO_Y(-1)$ and hence a short exact sequence
\[
 0 \to \cO_Y \to \cE \to \cO_Y(-1) \to 0.
\]
Since $\Ext^1(\cO_Y(-1),\cO_Y)\cong H^1(\bP^{n-1},\cO(1))=0$, the above exact sequence splits. So $\cE\cong\cO_Y\oplus\cO_Y(-1)$ and $E$ corresponds to the second projection $\cO_Y\oplus\cO_Y(-1)\twoheadrightarrow\cO_Y(-1)$. As a result,
$\hX$ is isomorphic to the blow up of $\bP^n$ at one point with $E$ corresponding to the exceptional divisor. Therefore, $X\cong\bP^n$.\qed
\end{emp}

The following is an application of Theorem \ref{mainthm} to Ding-semistable $\bQ$-Fano varieties with maximal volume (see \cite{fuj15, liu16} for backgrounds). This improves Fujita's result on the equality case in \cite[Theorem 5.1]{fuj15}. We remark that a different proof is presented in \cite[Proof 2 of Theorem 36]{liu16}.

\begin{thm}\label{ding-pn}
 Let $X$ be a Ding-semistable $\bQ$-Fano variety of dimension $n$. If $((-K_X)^n)\geq (n+1)^n$, then $X\cong\bP^n$.
\end{thm}

\begin{proof}
 Notice that $((-K_X)^n)\leq (n+1)^n$
 by \cite[Corollary 1.3]{fuj15}. Thus we have $((-K_X)^n)=(n+1)^n$.
 Let $p\in X$ be a smooth point. From \cite[Proof of 5.1]{fuj15}, we see that $\epsilon(-K_X,p)=n+1$. Hence $X\cong\bP^n$ by Theorem \ref{mainthm}.
\end{proof}

\section{Equality case}\label{sec3}

In this section we prove Theorem \ref{thm:equality}. Let $X$ be an $n$-dimensional $\bQ$-Fano variety with a smooth point $p\in X$.
Assume $\epsilon(-K_X,p)=n$. Following the proof of Theorem \ref{mainthm}, we have
that $D=\sigma^*(-K_X)-n E$ is semiample on $\hX$ and induces the morphism $g:\hX\to Y$. We now separate into two cases base on different behavior of $g$.

\subsection{\texorpdfstring{$g$}{g} is birational}\label{sec3.1}

\begin{lem} \label{lem:g(E)nef}
If $g:\hat{X}\rightarrow Y$ is birational, then $g|_E$ is a closed embedding, $-(K_{Y}+g(E))$ is ample and $g(E)\cong\bP^{n-1}$ is a nef divisor in the smooth locus of $Y$. Moreover, $Y$ is a $\bQ$-Fano variety.
\end{lem}

\begin{proof}
We see that $mD-E-K_{\hX}=(m-1)D$ is nef and big, so Kawamata-Viehweg vanishing implies that $g|_E:E\to Y$ is a closed embedding as in the proof of Theorem \ref{mainthm}. Hence $g(E)\cong E\cong\bP^{n-1}$. By Lemma \ref{birlocal}, it lies in the smooth locus of $Y$.

Since $g$ is induced by $D$, $-(K_{Y}+g(E))=\pi_* D$ is ample. To show the nefness of $g(E)$ 
we only need to show that $(L\cdot g(E))\geq 0$ for a line $L$ in $g(E)$. We may assume $L$ intersects the the exceptional locus of $g$. Denote by $L'$ the strict transform of $L$ in $\hX$. Let $W=g^*g(E)-E$, then it is an effective Cartier divisor supported on $\Ex(g)$. Since $-W\sim_{g.\bQ.}-K_{\hX}$ is $g$-ample, we have Ex($g$)$\subseteq W$, hence $(L'\cdot W)\ge1$ and $(L\cdot g(E))=(L'\cdot(E+W))=-1+(L'\cdot W)\ge0$.
\end{proof}

According to Lemma \ref{lem:g(E)nef}, we are now in the situation of Lemma \ref{birlocal} with $\lambda=1$. In order to classify $X$,
we first need to study the structure of the birational map $g:\hX\to Y$ in greater detail. This is accomplished by the following lemma.

\begin{lem}\label{lem:blowup}
Under the same notations and assumptions as in Lemma \ref{birlocal}, $\hX$ is the blowup of $Y$ along a divisor in $G$.
\end{lem}

\begin{proof}
First note that by Lemma \ref{birlocal} and its proof, $\hX$ has only compound Du Val singularities along Ex($g$), hence after shrinking $\hX$ and $Y$ we may assume that $\hX$ has only klt singularities.

Let $W=g^{*}G-E$ as above, then $W$ is $g$-exceptional and $-W$ is
a $g$-ample Cartier divisor on $\hat{X}$, hence we have $\hat{X}\cong\mathrm{Proj}\oplus_{m=0}^{\infty}\mathcal{J}_{m}$
where $\mathcal{J}_{m}=g_{*}\mathcal{O}_{\hat{X}}(-mW)$($m=0,1,\cdots$).
It is clear that each $\mathcal{J}_{m}$ is an ideal sheaf on $Y$.
Let $\mathcal{J}=\mathcal{J}_{1}$, we claim that $\mathcal{J}$ is
the ideal sheaf of a hypersurface in $g_{*}E$ and $\mathcal{J}_{m}=\mathcal{J}^{m}$.

To see this, note that since $-mW-K_{\hat{X}}\sim_{g.\mathbb{Q}}(m+1)E$
is $g$-ample and $\hat{X}$ is klt, we have $R^{1}g_{*}\mathcal{O}_{\hat{X}}(-mW)=0$
for all $m\ge0$. Hence from the pushforward $g_{*}$ of
\[
0\rightarrow\mathcal{O}_{\hat{X}}(-g^{*}G-mW)\rightarrow\mathcal{O}_{\hat{X}}(-(m+1)W)\rightarrow\mathcal{O}_{E}(-(m+1)W)\rightarrow0
\]
we obtain an exact sequence
\[
0\rightarrow\mathcal{J}_{m}(-G)\rightarrow\mathcal{J}_{m+1}\rightarrow\mathcal{O}_{E}(-(m+1)W)\rightarrow0
\]
Taking $m=0$, by Nakayama lemma we see that locally $\mathcal{J}=(a,b)$
is the ideal sheaf of $g(W)$ where $a=0$ (resp. $a=b=0$) is the
local defining equation of $G$ (resp. $g(W)$). Note that the restriction of $g$ to $E$ is an isomorphism, so $g(W)\cong W\cap E$ is a divisor (not necessarily irreducible or reduced) in $G$. Suppose we have
shown $\mathcal{J}_{m}=\mathcal{J}^{m}$ for some $m\ge1$ (the case
$m=1$ being clear), then the above exact sequence tells us that $\mathcal{J}_{m+1}$
is generated by $a\cdot\mathcal{J}_{m}$ and $b^{m+1}$, hence $\mathcal{J}_{m+1}=\mathcal{J}^{m+1}$
as well. The claim then follows by induction on $m$ and the lemma
follows immediately from the claim.
\end{proof}

Now we will classify the pairs $(Y,g(E))$ satisfying the statement of Lemma \ref{lem:g(E)nef}.
By abuse of notation, we will simply denote the divisor by $E$ instead of $g(E)$. We remark that Bonavero, Campana and Wi\'sniewski classified such pairs in \cite{blowupFano} when $Y$ is smooth. 

\begin{lem} \label{lem:Y}
Let $Y$ be an $n$-dimensional $\mathbb{Q}$-Fano variety containing
a prime divisor $E\cong\mathbb{P}^{n-1}$ in its smooth locus.
\begin{enumerate}
\item If $\rho(Y)=1$, then either $Y$ is a weighted projective space $\mathbb{P}(1^{n},d)$
for some $d\in\mathbb{Z}_{>0}$ and $E$ the hyperplane defined by
the vanishing of the last coordinate, or $n=2$, $Y\cong\mathbb{P}^{2}$
and $E$ is a smooth conic curve;
\item If $\rho(Y)\ge2$ and $-(K_{Y}+E)$ is ample, then $Y$ is a $\mathbb{P}^{1}$-bundle
$\mathbb{P}(\mathcal{O}\oplus\mathcal{O}(-d))$ over $\mathbb{P}^{n-1}$
for some $d\in\mathbb{Z}_{\ge0}$ and $E$ is a section. If $n\ge 3$
and $d\geq n$ then $E$ is the only section with negative normal bundle.
\end{enumerate}
\end{lem}

\begin{proof}
Note that in the case $\rho(Y)=1$, $E$ is necessarily an ample divisor
on $Y$. As $E$ does not intersect the singular locus of $Y$, $Y$
has only isolated singularities. By adjunction $-(K_{Y}+E)|_{E}=-K_{E}$
is ample, hence $-(K_{Y}+E)$ is ample as well. Let $Y^{\circ}$ be
the smooth locus of $Y$ and $i:E\rightarrow Y^{\circ}$ the inclusion.

First assume $\rho(Y)=1$ and $n\ge3$. By the generalized version
of Lefschetz hyperplane theorem \cite[Theorem II.1.1]{gm}, $H_{i}(Y^{\circ},E,\mathbb{Z})=H^{i}(Y^{\circ},E,\mathbb{Z})=0$
for $i<n$, hence by the universal coefficient theorem, $H^{n}(Y^{\circ},E,\mathbb{Z})$
is torsion free. As $n\ge3$, this implies the restriction map $i^{*}:H^{2}(Y^{\circ},\mathbb{Z})\rightarrow H^{2}(E,\mathbb{Z})$
is injective and has torsion free cokernel. But $H^{2}(E,\mathbb{Z})\cong\mathbb{Z}$
since $E\cong\mathbb{P}^{n-1}$, so $i^{*}$ is in fact an isomorphism.
As $Y$ is $\bQ$-Fano we have $H^1(Y,\cO_Y)=0$ by Kawamata-Viehweg vanishing and $Y$ is Cohen-Macaulay by \cite[Theorem 5.22]{km98}. Since $Z=\mathrm{Sing}Y$ consists of isolated points and $n\ge3$, by the long exact sequence of cohomology with support
\[\cdots\rightarrow H^1_Z(Y,\cO_Y)\rightarrow H^1(Y,\cO_Y)\rightarrow H^{1}(Y^{\circ},\mathcal{O}_{Y^{\circ}})\rightarrow H^2_Z(Y,\cO_Y) \rightarrow\cdots\]
we get $H^{1}(Y^{\circ},\mathcal{O}_{Y^{\circ}})=0$. Combining this with the exponential sequence $0\rightarrow\mathbb{Z}\rightarrow\mathcal{O}_{Y^{\circ}}\rightarrow\mathcal{O}_{Y^{\circ}}^{*}\rightarrow0$, we see that the restriction $i^{*}:\mathrm{Cl}(Y)=\mathrm{Pic}(Y^{\circ})\rightarrow\mathrm{Pic}(E)\cong\mathbb{Z}$ is also an isomorphism.

Let $H$ be the ample generator of $\mathrm{Cl}(Y)$, then $E\sim dH$
for some $d\in\mathbb{Z}_{>0}$. Let $\pi:Y'\rightarrow Y$ be the
(normalization of the) cyclic cover of degree $d$ of $Y$ ramified
at $E$ and $E'=\pi^{-1}(E)_{\mathrm{red}}$. Then $K_{Y'}+E'=\pi^{*}(K_{Y}+E)$
as $E$ is the only branched divisor, hence $Y'$ is also $\mathbb{Q}$-Fano
and $E'$ satisfies the same assumptions of the lemma. We also have
$\mathcal{O}_{E'}(dE')\cong\mathcal{O}_{E'}(\pi^{*}E)=\pi^{*}N_{E/Y}\cong\mathcal{O}_{E'}(d)$,
hence $N_{E'/Y'}\cong\mathcal{O}_{E'}(1)$ is the hyperplane class.
Note that $E'$ is ample since it's the preimage of the ample divisor
$E$. It now follows from the long exact sequence
\[
0\rightarrow H^{0}(Y',\mathcal{O}_{Y'})\rightarrow H^{0}(Y',\mathcal{O}_{Y'}(E'))\rightarrow H^{0}(E',N_{E'/Y'})\rightarrow H^{1}(Y',\mathcal{O}_{Y'})=0
\]
that the linear system $|E'|$ is base point free, has dimension $n$
and defines an isomorphism $Y'\cong\mathbb{P}^{n}$ such that $E'$
is mapped to a hyperplane. Our original pair $(Y,E)$ is then obtained
by taking a cyclic quotient of degree $d$ ramified at $E'$, and
is easily seen to be as claimed in the statement of the lemma.

Next assume $\rho(Y)=1$ and $n=2$. Then $Y$ has quotient singularity
and is $\mathbb{Q}$-factorial, hence $\mathrm{Cl}(Y)$ has rank one.
As $E$ is ample, $\pi_{1}(E)\rightarrow\pi_{1}(Y^{\circ})$ is surjective by \cite[Theorem II.1.1]{gm},
but $\pi_{1}(E)=\pi_{1}(\mathbb{P}^{1})=0$, so $Y^{\circ}$ is simply
connected as well. In particular, $\mathrm{Cl}(Y)=\mathrm{Pic}(Y^{\circ})$
is torsion-free and thus $\cong\mathbb{Z}$. Let $r$ be the index
of $i^{*}\mathrm{Cl}(Y)$ in $\mathrm{Pic}(E)$. As $-(K_{Y}+E)|_{E}=-K_{E}$
has degree 2, $r=1$ or 2. Let $H$ be the ample generator of $\mathrm{Cl}(Y)$,
then $(H.E)=r$ and $E\sim dH$ for some $d\in\mathbb{Z}_{>0}$. Let
$\pi:Y'\rightarrow Y$ be the corresponding cyclic cover of degree
$d$ and define $E'$ as before. By the same argument as the $n\ge3$
case, we have $N_{E'/Y'}\cong\mathcal{O}_{E'}(r)$, and if $r=1$,
the linear system $|E'|$ defines an isomorphism $(Y',E')\cong(\mathbb{P}^{2},\mathrm{hyperplane})$,
while if $r=2$, then linear system $|E'|$ embeds $Y'$ into $\mathbb{P}^{3}$
as a quadric surface. Taking cyclic quotients, we see that the origin
$(Y,E)$ is again as claimed. 

Finally assume $\rho(Y)\ge2$ and $-(K_{Y}+E)$ is ample. Let $l$ be a line in $E$. We claim that there is an extremal ray $\bR_{\geq 0}[\Gamma]$ in $\overline{NE}(Y)$ with $\Gamma$ an irreducible reduce curve on $Y$ such that $[\Gamma]\not\in\bR_{\geq 0}[l]$ and $(E\cdot \Gamma)>0$. If $(E\cdot l)=0$, then such $\Gamma$ exists since $E$ is not numerically trivial.
 If $(E\cdot l)>0$, consider the exact sequence
\[
0\to T_E|_l\to T_Y|_l \to N_{E/Y}|_l\to 0.
\]
It is clear that $T_E|_l$ is ample because $E\cong\bP^{n-1}$. On the other hand, $\deg N_{E/Y}|_l=(E\cdot l)>0$ hence $N_{E/Y}|_l$ is ample.
Therefore, $T_Y|_l$ is also ample which implies that $l$ is a very free rational curve in $Y$. Since $\rho(Y)\geq 2$, $\bR_{\geq 0}[l]$ cannot be an extremal ray of $\overline{NE}(Y)$ (otherwise the contraction of $l$ will contract $Y$ to a single point), which means that such $\Gamma$ exists. 

Now let $h:Y\to Z$ be the contraction of $\Gamma$. As we argued in the proof of Theorem \ref{mainthm}, $h|_E:E\to Y$ is a closed embedding, hence $(E\cdot \Gamma)=1$. Since $-(K_Y+E)$ is ample, we have $(-K_Y\cdot \Gamma)>1$. Then by the same reason in the proof of Theorem \ref{mainthm}, we conclude that $h$ has to be a fiber type contraction. Hence $Y$ is a $\bP^1$-fibration over $Z\cong \bP^{n-1}$ admitting a section $h|_E^{-1}:Z\to E$, so $Y\cong\bP_{Z}(\cO\oplus\cO(-d))$ with 
$d\geq 0$. If $n\geq 3$, then $E$ corresponds to either a surjection $\cO\oplus\cO(-d)\twoheadrightarrow\cO$ or a surjection $\cO\oplus\cO(-d)\twoheadrightarrow\cO(-d)$. If in addition $d\geq n$, then 
$-(K_Y+E)$ being ample implies that $E$ is the unique section corresponding
to the second projection $\cO\oplus\cO(-d)\twoheadrightarrow\cO(-d)$.
\end{proof}

Combining the last two lemmas we can give a partial classification of $X$:

\begin{lem}\label{lem:birational}
If $g$ is birational then $X$ is one of the following:
\begin{enumerate}
\item a degree $d+1$ weighted hypersurface $X_{d+1}=(x_0x_{n+1}=f(x_1,\cdots,x_n))\subset\mathbb{P}(1^{n+1},d)$;
\item the blow-up of $\bP^n$ along the complete intersection of a hyperplane and a hypersurface of degree $d\le n$;
\item a Gorenstein log Del Pezzo surface of degree $\ge 5$.
\end{enumerate}
\end{lem}

\begin{proof}
By Lemma \ref{lem:Y}, we have the following cases:

(1) $Y\cong\mathbb{P}(1^{n},d)$ with homogeneous coordinate $[y_{0}:\cdots:y_{n}]$
and $g(E)=(y_{n}=0)$. We have $N_{g(E)/Y}\cong\mathcal{O}_{E}(d)$.
By Lemma \ref{lem:blowup}, $\hat{X}$ is obtained by blowing up a
hypersurface $S=(f=0)$ in $g(E)$ where $f$ is a homogeneous polynomial
in $y_{0},\cdots,y_{n-1}$. As $N_{E/\hat{X}}\cong\mathcal{O}_{E}(-1)$
we see that $\deg f=d+1$. Consider the rational map $\phi:Y\dashrightarrow\mathbb{P}(1^{n+1},d)$
given by
\[
[y_{0}:\cdots:y_{n}]\mapsto[x_{0}:\cdots:x_{n+1}]=[y_{n}:y_{0}:\cdots:y_{n-1}:\frac{f(y_{0},\cdots,y_{n-1})}{y_{n}}]
\]
whose image lies in the weighted hypersurface $X_{d+1}$ define by
$x_0x_{n+1}=f(x_1,\cdots,x_n)$. It is clear that $\phi$
is contracts $g(E)$ to the point $[0:\cdots:0:1:0]$ and the indeterminacy
locus of $\phi$ is exactly $S$. By inspecting each affine chart
$(x_{i}\neq0)\subset Y$ it is easy to see that after blowing up $S$,
$\phi$ extends to a birational morphism $\hat{X}\rightarrow X_{d+1}$
that contracts $E$, hence $X\cong X_{d+1}$ as in the first case
in the statement of the lemma.

(2) $Y$ is a $\mathbb{P}^{1}$-bundle $\mathbb{P}(\mathcal{O}\oplus\mathcal{O}(-d))$
over $\mathbb{P}^{n-1}$ $(n\ge3$) and $g(E)$ is a section. Since
$g(E)$ is nef by Lemma \ref{lem:g(E)nef}, we have $d<n$ by Lemma
\ref{lem:Y}. Going back to the last part of the proof of Lemma \ref{lem:Y}
we see that the section $g(E)$ corresponds to a surjection $\mathcal{O}\oplus\mathcal{O}(-d)\twoheadrightarrow\mathcal{O}$
and hence $N_{g(E)/Y}\cong\mathcal{O}_{E}(d)$. By Lemma
\ref{lem:blowup} as in previous case, $\hat{X}$ is obtained by blowing
up a hypersurface $S$ of degree $d+1$ in $g(E)$. It is straightforward
to see that the elementary transformation of $Y$ with center $S$
is the $\mathbb{P}^{1}$-bundle $\mathbb{P}(\mathcal{O}\oplus\mathcal{O}(-1))$
over $\mathbb{P}^{n-1}$ , which is isomorphic to the blowup of a
point $R$ on $\mathbb{P}^{n}$, such that the strict transform $E'$
(resp. $H$) of $g(E)$ (resp. the negative section on $Y$) becomes
the exceptional divisor over $R$ (resp. a hyperplane in $\mathbb{P}^{n}$
that is disjoint from $R$). Contracting $E'$ and reversing this
procedure we see that $X$ is the blowup of $\mathbb{P}^{n}$ along
a hypersurface of degree $d+1\le n$ in a hyperplane.

(3) $Y\cong\mathbb{P}^{2}$ and $g(E)$ is a smooth conic, or $Y$ is
a ruled surface over $\mathbb{P}^{1}$ and $g(E)$ is a section. In
either case $Y$ is smooth and $\hat{X}$ is obtained by blowing up
subschemes of $g(E)$. Locally on $Y$, such a subscheme is defined
by $(a=b^{k}=0)$ where $a,b$ are local coordinates such that $g(E)=(a=0)$.
$\hat{X}$ then has local equation $at=b^{k}$ or $a=b^{k}t$ and
it follows that both $\hat{X}$ and $X$ have only Du Val singularities of type $A$. As $D=\sigma^{*}(-K_{X})-2E$ is big and nef and Cartier
in this case we have $(K_{X}^{2})=(D^{2})-4(E^{2})=(D^{2})+4\ge5$,
so $X$ is as described in the third case of the statement of the
lemma.
\end{proof}

\subsection{\texorpdfstring{$g$}{g} is of fiber type} \label{sec3.2}

\begin{lem}
If $g$ is of fiber type, then every fiber has dimension $1$, $g|_E:E\rightarrow Y$ is a double cover and $-K_{\hX}\sim_{g.\bQ.}E$ is $g$-ample.
\end{lem}

\begin{proof}
Since $\epsilon(-K_X,p)>n-1$, $\hX$ is $\bQ$-Fano, so $-K_{\hX}\sim_{g.\bQ.}E$ is $g$-ample. $D|_E$ is ample, so $E\rightarrow Y$ is finite and every fiber of $g$ has dimension one.
Let $l$ be a general fiber, then $l\cong\mathbb{P}^{1}$ and $(-K_{\hat{X}}\cdot l)=2=(E\cdot l)$,
so $E$ is a double section.
\end{proof}

Similar to the previous case, we first analyze the local structure of $g$ in a slightly more general setting. For ease of notations, we call $g:\hX\rightarrow Y$ (where $\hX$ and $Y$ are normal quasi-projective varieties) a \emph{rational conic bundle} if $g$ is proper, every fiber of $g$ has dimension $1$ and the generic fiber is isomorphic to $\bP^1$. If in addition $\hX$ is Cohen-Macaulay and there exists a Cartier divisor $E$ on $\hX$ such that $-K_{\hX}\sim_{g.\bQ.}E$ is $g$-ample, then we say that the rational conic bundle is \emph{Gorenstein}.
It is clear that a conic bundle is automatically a Gorenstein rational conic bundle.

\begin{lem}\label{fiblocal2}
Let $g:S\rightarrow C$ be a Gorenstein rational conic bundle. Assume $\dim S=2$, then $S$ is a conic bundle and in particular has only Du Val singularities.
\end{lem}

\begin{proof}
Let $l$ be an irreducible component of a fiber of $g$, then $(-K_S\cdot l)=(E\cdot l)$ is a positive integer since $E$ is Cartier and $-K_S$ is $g$-ample. On the other hand, if $F$ is a general fiber of
$g$ then $(-K_S\cdot F)=2$. Hence every fiber of $g$ has at most
two irreducible components (counting multiplicities), so on the minimal
resolution of $S$ (which is a birationally ruled surface over $C$), every fiber over $C$ has one of the following as its dual graph:
\[(-2)-(-1)-(-2),\]
\[(-1)-(-2)-(-2)-\cdots-(-2)-(-1),\]
or
\begin{center}
\begin{tikzpicture}
\node (0) [anchor = west] {$(-2)$};
\node (1) [right = 0.25cm of 0, anchor = west] {$(-2)-\cdots-(-2)-(-1)$};
\node (2) [above left = 0.5cm of 0, anchor = east] { $(-2)$};
\node (3) [below left = 0.5cm of 0, anchor = east] { $(-2)$};
\path
    (0) edge (1)
    (2) edge (0)
    (3) edge (0);
\end{tikzpicture}
\end{center}
As $S$ is obtained by contracting those $(-2)$-curves, it has only Du Val singularities and is a conic bundle.
\end{proof}

\begin{cor}\label{fiblocal}
If $g:\hX\rightarrow Y$ is a Gorenstein rational conic bundle such that $Y$ is smooth, then $\hX$ is a conic bundle over $Y$. In particular, $\hX$ is a hypersurface in $\bP(\cE)$ for some rank $3$ vector bundle $\cE$ on $Y$.
\end{cor}

\begin{proof}
Let $y\in Y$ and $C$ a general complete intersection curve on $Y$ passing through $y$. Let $S=\hX\times_Y C$. Since $\hX$ is Cohen-Macaulay, $S$ is $S_2$. From the proof of Lemma \ref{fiblocal2} we know that the fiber $g^{-1}(y)$ has at most 2 irreducible components (counting multiplicities), hence $S$ is smooth at every generic point of $g^{-1}(y)$, for otherwise $g^{-1}(y)$ contains a component of multiplicity $\ge2^2=4$. It follows that $S$ is normal. By adjunction it is easy to see that $S$ is a Gorenstein rational conic bundle over $C$, so by Lemma \ref{fiblocal2}, $S$ has only Du Val singularities and is a conic budle, hence every fiber of $g$ is isomorphic to a conic and $\hX$ has cDV singularities which is Gorenstein. The lemma then follows from standard arguments (see e.g. \cite[Theorem 7]{cut}).
\end{proof}

Unfortunately in our classification problem, the Gorenstein rational conic bundle $g:\hX\rightarrow Y$ does not have a smooth base. Nevertheless, there is a smooth double section $E$. Hence we would like to apply Corollary \ref{fiblocal} to $\tg:\tX\rightarrow\tY$, where $\tY\cong E$ and $\tX$ is the normalization of $\hat{X}\times_{Y}\tY$. For this purpose, we need to show that $\tX$ is Gorenstein rational conic bundle over $\tY$. This is given by the following lemma.

\begin{lem}\label{basechange}
Let $g:\hX\rightarrow Y$ be a Gorenstein rational conic bundle and $\phi:\tY\rightarrow Y$ a finite morphism between normal varieties. Let $\tX$ be the normalization of $\hX\times_{Y}\tY$. Assume that $\hX$ has klt singularities and the branch divisor of $\phi$ is disjoint from the singular locus of $\tY$ and $Y$. Then $\tg:\tX\rightarrow\tY$ is also a Gorenstein rational conic bundle.
\end{lem}

\begin{proof}
By shrinking $Y$ we may assume either $\phi$ is unramified in codimension one or both $Y$ and $\tY$ are smooth. In the first case $\tX$ is also klt by \cite[Proposition 5.20]{km98} hence is CM, and the other properties of Gorenstein rational conic bundles are preserved by a finite base change that is \'etale in codimension one. In the second case $g$ is a conic bundle by Lemma \ref{fiblocal}, hence the same holds for $\tg$.
\end{proof}

The pullback $E'$ of $E$ to $\tX$ is then a union of two sections $E_1$ and $E_2$. If they are disjoint, we have a simple description of the conic bundle $\tg$:

\begin{lem}\label{disjoint}
Let $\tg:\tX\rightarrow\tY$ be a conic bundle with smooth base. Assume that there are two disjoint sections $E_1$ and $E_2$ that are Cartier as divisors on $\tX$ and such that $-K_{\tX}\sim_{g.\bQ.}E_1+E_2$. Then there is a birational morphism $u:\tX\rightarrow Z=\bP_{\tY}(\cO\oplus\cL)$ \emph{(}where $\cL\cong N_{E_1/\tX}$\emph{)} sending $E_1$, $E_2$ to two disjoint sections $E_1'$, $E_2'$ of $Z$ such that $\tX$ is the blow up of $Z$ along a divisor in $E_2$.
\end{lem}

\begin{proof}
If every fiber of $\tg$ is an irreducible $\bP^1$ then $\tX\cong\bP_{\tY}(\cO\oplus\cL)$ and there is nothing to prove. So we may assume $l=l_1+l_2$ is a reducible fiber. We have $(E_1+E_2\cdot l_j)=(-K_{\tX}\cdot l_j)=1$ ($j=1,2$). Since the section $E_i$ is Cartier, we have $(E_i\cdot l_j)=\delta_{ij}$ after rearranging indices. Let $u:\tX\rightarrow Z$ be the contraction of the extremal ray $\bR_+[l_2]$ and let $E_1'$, $E_2'$ be strict transform of $E_1$, $E_2$. As $E_i$ is a section of $\tg$ and $E_i\rightarrow\tY$ factors through $E_i'$, the restriction $u|_{E_i}$ is an isomorphism. In addtion we have $-(K_{\tX}+E_2)\sim_{u.\bQ.}0$ since its intersection number with $l_2$ is zero. Hence the lemma follows by a direct application of Lemma \ref{lem:blowup}.
\end{proof}

Putting everything together and specializing to $E\cong\bP^{n-1}$, we now finish the second part of the classification of $X$ with $\epsilon(-K_X,p)=n$.

\begin{lem}\label{lem:fibertype}
If $g$ is of fiber type then $X$ is one of the following:
\begin{enumerate}
\item a Gorenstein log del Pezzo surface of degree 4;
\item quotient of a quadric hypersurface in $\mathbb{P}^{n+1}$ by an involution that is fixed point free in codimension 1;
\item a quartic weighted hypersurface in $\mathbb{P}(1^n,2^2)$.
\end{enumerate}
\end{lem}

\begin{proof}
If $n=\dim X=2$ then by Lemma \ref{fiblocal2}, $\hX$ and hence $X$ has only Du Val singularities. We have $\sigma^*(-K_X)-2E\sim_{g.\bQ.}0$, so $(K_X^2)=-4(E^2)=4$ and we are in case (1). Hence in the remaining part of the proof we assume that $n\ge3$. 

We keep using the notations introduced in this subsection. Let $\tX\rightarrow\bar{X}$ be the Stein factorization
of the composition $\tX\rightarrow\hat{X}\rightarrow X$, then $\bar{X}\rightarrow X$ is a double cover. The double cover $E\rightarrow Y$ is either unramified in codimension one or the quotient $\bP^{n-1}\rightarrow\bP(1^{n-1},2)$ in which case the branch divisor is a hyperplane on $\bP^{n-1}$, so the conditions and conclusions of Lemma \ref{basechange} are satisfied and we see that $\tg:\tX\rightarrow\tY$ is a conic bundle over $\tY\cong\bP^{n-1}$ by Corollary \ref{fiblocal}.

If $h:\tX\rightarrow\hX$ is unramified in codimension one, so is $\bar{X}\rightarrow X$ and we have $\mathrm{codim}_{E_{1}\cap E_{2}}E_{i}$ $\ge2$. But since $\tX$
is Cohen-Macaulay and $E'=E_1+E_2$ is a Cartier divisor, $E_1\cup E_2$ is $S_{2}$.
It follows that $E_{1}$ and $E_{2}$ do not intersect at all, hence
they are disjoint smooth Cartier divisors in $\tX$ with normal
bundle $\mathcal{O}_{\bP^{n-1}}(-1)$. As $K_{\tX}+E_1+E_2=h^*(K_{\hX}+E)\sim_{g.\bQ.}0$, it follows from Lemma \ref{disjoint} that $\tX$ is a blowup of $Z\cong\mathbb{P}_{\tY}(\mathcal{O}\oplus\mathcal{O}(-1))\cong\mathrm{Bl}_z\bP^n$ along a hypersurface in the strict transform of a hyperplane. For the normal bundle to match, it is the blowup of a quadric hypersurface. As $\bar{X}$ is obtained by contracting $E_{1}\cup E_{2}$ from $\tX$, it is a quadric hypersurface in $\mathbb{P}^{n+1}$,
and $X$ is the quotient of $\bar{X}$ by an involution that acts fixed point free in codimension one as in case (2).

If $h:\tX\rightarrow\hX$ is ramified in codimension one, then it is ramified along $\tg^*H$ where $H$ is a hyperplane on $\tY$. As in the last paragragh $E_{1}\cap E_{2}$ has pure codimension one, so $E'$ is a union of
two $\mathbb{P}^{n-1}$ intersecting transversally at a hyperplane. The conic bundle $\tX$ is a hypersurface in some $\mathbb{P}(\mathcal{E})$ over $\tY$. To compute $\mathcal{E}$, first note that $-(K_{\tX}+E')=\tg^*M$ for some $M\in\mathrm{Pic}(E)$ since it restricts to a trivial bundle on every fiber of $\tg$; we also have $-(K_{\tX}+E')|_{E'}=-K_{E'}=(n-1)\tg^{*}H$,
so $M\sim(n-1)H$. Combining with $N_{E'/\tX}\cong\tg^{*}\mathcal{O}_{\tY}(-H)$ we have $-K_{\tX}|_{E'}\cong\tg^{*}(n-2)H$.
Now apply $\tg_{*}$ to the exact sequence 
\[
0\rightarrow\mathcal{O}_{\tX}(-K_{\tX}-E)\rightarrow\mathcal{O}_{\tX}(-K_{\tX})\rightarrow\mathcal{O}_{E'}(-K_{\tX})\rightarrow0
\]
we obtain another exact sequence \[0\rightarrow\mathcal{O}_{\tY}((n-1)H)\rightarrow\tg_{*}\mathcal{O}_{\tX}(-K_{\tX})\rightarrow\mathcal{O}_{\tY}((n-2)H)\oplus\mathcal{O}_{\tY}((n-3)H)\rightarrow R^{1}\tg_{*}\mathcal{O}_{\tX}\otimes M=0
\]
hence $\tg_{*}\mathcal{O}_{\tX}(-K_{\tX})\cong\oplus_{k=1}^{3}\mathcal{O}_{\tY}((n-k)H)$
and we may choose $\mathcal{E}\cong\oplus_{k=0}^{2}\mathcal{O}_{\tY}(kH)$.
Let $\pi$ be the projection $\mathbb{P}(\mathcal{E})\rightarrow\tY$ and $\mathcal{O}_{\mathbb{P}(\mathcal{E})}(1)$ the
relative hyperplane class. $\tX$ corresponds to section of
$\mathcal{O}_{\mathbb{P}(\mathcal{E})}(2)\otimes\pi^{*}\mathcal{O}_{\tY}(mH)$
for some $m\in\mathbb{Z}$ and by adjunction formula we have $\mathcal{O}_{\tX}(-K_{\tX})\cong\mathcal{O}_{\tX}(1)\otimes\tg^{*}\mathcal{O}_{\tY}((n-3-m)H)$, hence  $\tg_{*}\mathcal{O}_{\tX}(-K_{\tX})\cong\mathcal{E}\otimes\mathcal{O}_{\tY}((n-3-m)H)$.
Comparing this to the previous formula for $\tg_{*}\mathcal{O}_{\tX}(-K_{\tX})$
we see that $m=0$. The surjection $\mathcal{E}\twoheadrightarrow\mathcal{O}_{\tY}$
defines a section $S$ of $\mathbb{P}(\mathcal{E})\rightarrow \tY$
that is disjoint with $\tX$ (since $\mathcal{O}_{\mathbb{P}(\mathcal{E})}(2)|_{S}\cong\mathcal{O}_{S}$)
and the linear projection from $S$ makes $\tX$ into a double
cover over the $\mathbb{P}^{1}$-bundle $\mathbb{P}_{\tY}(\mathcal{O}(H)\oplus\mathcal{O}(2H))$,
which is also the blowup of a point on $\mathbb{P}^{n}$, such that
$E'$ is mapped to the exceptional divisor and $\tg^{*}H$ to the
strict transform of a hyperplane passing through the center of blowup.
$\bar{X}$ is then a double cover of $\mathbb{P}^{n}$, and as $-(K_{\tX}+E')\sim(n-1)\tg^{*}H$
we have $-K_{\bar{X}}\sim(n-1)\tau^{*}H'$ where $H'$ is a hyperplane
on $\mathbb{P}^{n}$ and $\tau:\bar{X}\rightarrow\mathbb{P}^{n}$
the double cover. It follows that $\bar{X}$ is a weighted hypersurface of degree 4 in
$\mathbb{P}(1^{n+1},2)$. The original $X$ is then obtained as the
quotient of $\bar{X}$ by an involution that fixes a hyperplane section (i.e. the strict transform of $\tg^*H$), hence is a quartic weighted hypersurface in $\mathbb{P}(1^n,2^2)$ as in case (3).
\end{proof}

\begin{emp}[\emph{Proof of Theorem \ref{thm:equality}}]
By Lemma \ref{lem:birational} and \ref{lem:fibertype}, we have the following five possibilities for $X$. Note that by Theorem \ref{mainthm} it suffices to show that $\epsilon(-K_{X},p)\ge n$ in each case.

(1) $X\cong X_{d+1}=(x_0x_{n+1}=f(x_1,\cdots,x_n))\subseteq\mathbb{P}(1^{n+1},d)$. If $d=1$ then
$X$ is a quadric hypersurface and the result is clear (or see case (4)). Otherwise $d>1$ and we have $q=[0:\cdots:0:1]\in X$. Let $p$ be
a smooth point on $X$ and let $\sigma:Z\rightarrow\mathbb{P}(1^{n+1},d)$
be the blowup of $\mathbb{P}(1^{n+1},d)$ at $p$ with exception divisor
$V$. Let $H$ be the divisor class $\mathcal{O}(1)$ on $\mathbb{P}(1^{n+1},d)$,
then we have $\sigma^{*}(-K_{X})-nE=n(\sigma^{*}H-V)|_{\hat{X}}$.
The base locus of the linear system $|\sigma^{*}H-V|$ on $Z$ is
the strict transform of the line $l$ joining $p$ and $q$. For general choice of $p$ we have $l\not\subseteq X$, hence $\sigma^{*}(-K_{X})-nE$ is nef on $\hat{X}$, yielding $\epsilon(-K_{X},p)\ge n$. 

(2) $X$ is a quartic hypersurface in $\mathbb{P}(1^n,2^2)$. Up to weighted projective isomorphism we may assume that $X$ is defined by the equation $q(x_n,x_{n+1})+x_n h(x_0,\cdots,x_{n-1})=f(x_0,\cdots,x_{n-1})$ where $\deg q=\deg h=2$, $\deg f=4$ and $h=0$ if $q\neq ax_{n+1}^2$. Let $p\in X$ be a smooth point and define $H$, $V$ in the similar way as in the first case. We have $\sigma^{*}(-K_{X})-nE=n(\sigma^{*}H-V)|_{\hat{X}}$. The base locus of $|\sigma^{*}H-V|$ is the plane $\Sigma$ spanned by $p$ and the line $(x_0=\cdots=x_{n-1}=0)$, so $D$ is nef (i.e. $\epsilon(-K_{X},p)\ge n$) if and only if for every curve $C\subseteq\Sigma\cap X$ we have $(D\cdot C)\geq0$. It is easy to see that $\frac{1}{n}(D\cdot C)=\frac{1}{4}\deg C-\mult_p C$. As $\deg (\Sigma\cap X)\leq4$ we see that $(D.C)\geq0$ if and only if $\Sigma\cap X$ is an irreducible curve that is smooth at $p$. Suppose $p=[c_0:\cdots:c_{n+1}]$, then $\Sigma\cap X$ is given by the equation $q(y_1,y_2)+h(c_0,\cdots,c_{n-1})y_1 y_0^2=f(c_0,\cdots,c_{n-1})y_0^4$ in $\Sigma\cong\mathbb{P}(1,2,2)$. From this it is clear that $\epsilon(-K_{X},p)\ge n$ for general $p\in X$ if and only if $q$ is not a square or $hq\neq0$. After another change of variable we see that $X$ is a quartic hypersurface of the form $x_n x_{n+1}=f(x_0,\cdots,x_{n-1})$ or $x_{n+1}^2+x_n h(x_0,\cdots,x_{n-1})=f(x_0,\cdots,x_{n-1})$ ($h\neq0$). 

(3) $X$ is the blowup of a hypersurface $S$ of degree $d\le n$ in a hyperplane of $\mathbb{P}^{n}$. Let $V$ be the exceptional divisor over $S$, $H$ the pullback of $\mathcal{O}_{\mathbb{P}^{n}}(1)$
on $X$ and $H'\subset X$ the strict transform of the hyperplane containing $S$. Let $p\in X$ be a point outside $H'\cup V$. We have
$D=\sigma^{*}(-K_{X})-nE\sim \sigma^*H'+n(\sigma^*H-E)$. We want to show that $D$
is nef. Since $\sigma^*H-E$ is already nef, it remains to show that $(D\cdot l)>0$
where $l$ is a line in $\sigma^*H'$. Then a direct computation shows that
$(D\cdot l)=(-K_{X}\cdot l)=(((n+1)H-V)\cdot l)=n+1-d>0$. Thus $D$ is nef and $\epsilon(-K_{X},p)\ge n$.

(4) $X=Q/\tau$ where $Q$ is a quadric hypersurface and $\tau\in\mathrm{Aut}(Q)$
an involution that is fixed point free in codimension one. Let $p_{1}$ be a smooth point of $Q$, let $p_{2}=\tau(p_{1})$
and $p$ be their image in $X$. Let $\psi:\hat{Q}\rightarrow Q$
be the blowup of $p_{1}$ and $p_{2}$ with exceptional divisors $E_{1}$
and $E_{2}$. Since $h:Q\rightarrow X$ is \'etale in codimension
one, the divisor $D=\sigma^{*}(-K_{X})-nE$ pulls back to $D'=\psi^{*}(-K_{Q})-nE_{1}-nE_{2}=n(\psi^{*}H-E_{1}-E_{2})$
where $H$ is the hyperplane class on $Q$. Similar to case (1), $D'$
is the restriction of a line bundle (also denoted by $D'$) on blowup of
$\mathbb{P}^{n+1}$ at $p_{1}$, $p_{2}$ whose base locus is the strict transform of the line $l$ joining $p_{1}$
and $p_{2}$. We also have $(D'\cdot l)=-n<0$. Hence $D$ is nef and $\epsilon(-K_{X},p)\ge n$ if and only if $l\not\subseteq Q$. We may diagonalize $\tau$ and choose homogeneous coordinate $x_i$ so that $\tau(x_i)=\delta_i x_i$ where $\delta_i=\pm1$. It is then not hard to verify that $l\not\subseteq Q$ for general choice of $p$ if and only if $Q$ is given by the equation $\sum_{i=0}^k x_i^2=0$ for some $2\leq k\leq n+1$ such that $\delta_i$ take different values for $i=0,\cdots,k$.

(5) $X$ is a Gorenstein log Del Pezzo surface of degree $(K_{X}^{2})\ge4$.
We claim that if $S$ is a Gorenstein log Del Pezzo surface of degree
$d\ge3$, then there exists an irreducible curve $C\in|-K_{S}|$ with
a double point $p$ lying in the smooth locus of $S$. After blowing
up $d-3$ general points on $S$, it suffices to prove the claim when
$d=3$, in which case $S$ is a nodal cubic surface in $\mathbb{P}^{3}$
by \cite[Theorem 4.4]{anti-canonical}. But then there are only finitely many lines on $S$ whereas
by dimension count there exists $C\in|-K_{X}|$ that is singular at
any given $p\in S$, hence the claim follows immediately. Using such
$C\in|-K_{X}|$ and take $p=\mathrm{Sing}(C)$, we have $\sigma^{*}(-K_{X})-2E\sim C'$
where $C'$ is the strict transform of $C$ and $(C'^{2})=(K_{X}^{2})-4\ge0$,
hence $C'$ is nef and $\epsilon(-K_{X},p)\ge n=2$.

It remains to show that all $\bQ$-Fano varieties listed in the statement of Theorem \ref{thm:equality} have only klt singularities. From the equations there we see that the singularities of $X$ are always quotients of $cA$-type singularities that are \'etale in codimension $1$ (hence are klt by \cite[1.42]{mmp} and \cite[Proposition 5.20]{km98}) except when $X$ is a quartic hypersurface $x_{n+1}^2+x_nh=f$ in $\bP(1^n,2^2)$ and $x\in (x_n=x_{n+1}=0)\cap X$ satisfies $\mult_x h=2$ and $\mult_x f\ge 3$. In the latter case, we may assume $x=[1:0:\cdots:0]$ and locally $X$ is a double cover of $\bC^n$ ramified along $D=(x_nh=f)$. If $h$ is not a perfect square, then the pair $(\bC^n,D)$ degenerates to $(\bC^n,D_0)$ where $D_0=(x_nh=0)$ (consider the $\bC^*$-action $(x_1,\cdots,x_n)\mapsto(t^2x_1,\cdots,t^2x_{n-1},tx_n)$ for $t\neq0$). Clearly $(\bC^n,\frac{1}{2}D_0)$ is klt, so it follows from adjunction that $(\bC^n,\frac{1}{2}D)$ is also klt which implies that $X$ is klt by \cite[Proposition 5.20]{km98}. If $h$ is a perfect square, then by \cite[page 168]{km98} we know that $X$ is a cDV singularity which is klt as well.
\qed
\end{emp}

\section{Seshadri constants below \texorpdfstring{$n$}{n}}\label{sec4}

In this section, we prove Theorem \ref{ratsesh} using the following examples.

\begin{expl}
Let $X$ be the weighted projective space $\mathbb{P}(1,a_{1},\cdots,a_{n})$
where $a_{1}\le\cdots\le a_{n}$ are positive integers satisfying
$\gcd(a_1,\cdots,a_n)=1$. Let $p\in X$ be the smooth
point with coordinate $[1:0:\cdots:0]$. We claim that the Seshadri constant of $-K_{X}$ at $p$ is $\epsilon(-K_{X},p)=\frac{1}{a_{n}}(1+\sum_{i=1}^{n}a_{i})$. As before let $\sigma:\hat{X}\rightarrow X$ be the blowup of $X$ at $p$ and $E$ the exceptional divisor. Since $\hat{X}$ is a toric
variety, the torus invariant divisor $L_{x}=\sigma^{*}(-K_{X})-xE$
is nef if and only if it has non-negative intersection number with
all torus invariant lines, and as $-K_{X}$ is ample on $X$ and $E$
has ample conormal bundle, it suffices to check $(L_{x}\cdot l_{i})\ge0$
where $l_{i}$ is the strict transform of the line on $X$ joining
$p$ and the point whose only nonzero coordinate is at the $i$-th
entry ($i>0$). It is straightforward to compute $(L_{x}\cdot l_{i})=\frac{1}{a_{i}}(1+\sum_{i=1}^{n}a_{i})-x$,
so $\epsilon(-K_{X},p)=\frac{1}{a_{n}}(1+\sum_{i=1}^{n}a_{i})$.
Taking $a_1=\cdots=a_{m-1}=1$, $a_m=r-m$, $a_{m+1}=\cdots=a_n=s$ where $1\le m<n$ and $s\ge r>m$ we get $\epsilon(-K_{X},p)=n-m+\frac{r}{s}$, hence the Seshadri constant $\epsilon(-K_{X},p)$ can be any rational number in the interval $(1,n]$.
\end{expl}

\begin{expl}
More generally, let $X$ be the weighted projective space $\mathbb{P}(a_{0},\cdots,a_{n})$ where $a_0\le\cdots\le a_n$ have no common factor and $p\in X$ a smooth point on the line $l:x_2=\cdots=x_n=0$ (such $p$ exists exactly when $\gcd(a_0,a_1)=1$). We claim that $\epsilon(-K_{X},p)$ is the smaller one of $\frac{1}{a_{n}}\sum_{i=0}^{n}a_{i}$ and $\frac{1}{a_{0}a_{1}}\sum_{i=0}^{n}a_{i}$.  Indeed, since $X$ is toric and $p$ is invariant under an $(n-1)$-dimensional subtorus $T$, the Mori cone of $\hX=\mathrm{Bl}_pX$ is generated by a line in $E$ and the strict transform $\hat{C}$ of a curve $C\subseteq X$ containing $p$ that is invariant under the action of $T$. Hence $C$ is the line joining $p$ and a $T$-invariant point. For $D=\sigma^*(-K_X)-\delta E$, we have $(D\cdot\hat{C})=\frac{1}{a_{0}a_{1}}\sum_{i=0}^{n}a_{i}-\delta$ if $C=l$, otherwise $(D\cdot\hat{C})=\frac{1}{a_{j}}\sum_{i=0}^{n}a_{i}-\delta$ for some $j$. The claim then follows by setting $(D\cdot\hat{C})\ge0$.  Taking $a_0=s-1$, $a_1=\cdots=a_{n-1}=s$, $a_n=(r-1)(s-1)-(n-1)s$ with $s\ge r\gg 0$ we get $\epsilon(-K_{X},p)=\frac{r}{s}$, hence the Seshadri constant $\epsilon(-K_{X},p)$ can be any rational number in the interval $(0,1]$ as well.
\end{expl}

\begin{rem}
As the previous examples give some possible values of $\epsilon(-K_X,p)$, it is natural to ask whether these are all possible values. When $\epsilon(-K_X,p)\geq n-1$, the Rationality Theorem \cite[Theorem 3.5]{km98} implies that $\epsilon(-K_X,p)$ is necessarily a rational number. When $\epsilon(-K_X,p)< n-1$, it is not clear to us whether
$\epsilon(-K_X,p)$ is rational, although there are no known examples of irrational Seshadri constants according to \cite[Remark 5.1.13]{pos1}.
\end{rem}

\bibliography{ref}
\bibliographystyle{alpha}

\end{document}